\documentclass[12pt]{amsart}
\usepackage{amscd,verbatim,amssymb}
\usepackage[all]{xy}

\newcommand{\sA}{\mathcal{A}}
\newcommand{\sC}{\mathcal{C}}
\newcommand{\sD}{\mathcal{D}}
\newcommand{\sE}{\mathcal{E}}

\newcommand{\sO}{\mathcal{O}}
\newcommand{\sP}{\mathcal{P}}
\newcommand{\sQ}{\mathcal{Q}}
\newcommand{\sT}{\mathcal{T}}

\newcommand{\F}{\mathbf{F}}
\newcommand{\G}{\mathbf{G}}
\renewcommand{\H}{\mathbf{H}}
\newcommand{\N}{\mathbf{N}}
\newcommand{\T}{\mathbf{T}}
\newcommand{\V}{\mathbf{V}}
\newcommand{\Y}{\mathbf{Y}}
\newcommand{\Z}{\mathbf{Z}}

\newcommand{\fP}{\mathfrak{P}}

\newcommand{\bDelta}{\operatorname{\mathbf{\Delta}}}
\newcommand{\Cat}{\operatorname{\mathbf{Cat}}}
\newcommand{\Set}{\operatorname{\mathbf{Set}}}
\newcommand{\sSet}{\operatorname{\mathbf{sSet}}}
\newcommand{\Ab}{\operatorname{\mathbf{Ab}}}
\newcommand{\Tot}{\operatorname{Tot}}
\newcommand{\op}{{\operatorname{op}}}
\newcommand{\tf}{{\operatorname{tf}}}
\newcommand{\coh}{{\operatorname{coh}}}

\newcommand{\Ker}{\operatorname{Ker}}
\newcommand{\Simpl}{\operatorname{Simpl}}
\newcommand{\Coker}{\operatorname{Coker}}
\newcommand{\Spec}{\operatorname{Spec}}
\newcommand{\Aut}{\operatorname{Aut}}
\newcommand{\rg}{\operatorname{rk}}
\newcommand{\Sk}{\operatorname{Sk}}

\newcommand{\colim}{\varinjlim}

\renewcommand{\phi}{\varphi}
\renewcommand{\epsilon}{\varepsilon}

\newcommand{\by}[1]{\overset{#1}{\longrightarrow}}
\newcommand{\yb}[1]{\overset{#1}{\longleftarrow}}
\newcommand{\iso}{\by{\sim}}
\newcommand{\osi}{\yb{\sim}}

\swapnumbers
\newtheorem{prop}{Proposition}[subsection]
\newtheorem{lemma}[prop]{Lemma}
\newtheorem{thm}[prop]{Theorem}
\newtheorem{cor}[prop]{Corollary}
\theoremstyle{definition}
\newtheorem{defn}[prop]{Definition}
\newtheorem{para}[prop]{}
\theoremstyle{remark}
\newtheorem{rk}[prop]{Remark}
\newtheorem{ex}[prop]{Example}

\numberwithin{equation}{section}

\newcounter{spec}
\newenvironment{thlist}{\begin{list}{\rm{(\roman{spec})}}%
{\usecounter{spec}\labelwidth=20pt\itemindent=0pt\labelsep=10pt}}%
{\end{list}}%

\begin{document}

\title{Around Quillen's Theorem A}
\author{Bruno Kahn}
\address{Institut de Math\'ematiques de Jussieu\\Case 247\\
4 place Jussieu\\
75252 Paris Cedex 05\\
FRANCE}
\email{kahn@math.jussieu.fr}
\date{June 2014}
\begin{abstract}
We introduce a notion of cellular functor. It allows us to give a variant of Quillen's proof of the Solomon-Tits theorem, which does not use Theorem A. We also use it to generalise some exact sequences of Quillen, and to reformulate them into a rank spectral sequence converging to the
homology of the $K'$-theory space of an integral scheme.
\end{abstract}
\maketitle

\hfill In the memory of Daniel Quillen

\tableofcontents

\section*{Introduction} Let $A$ be a Dedekind domain. Inspection shows immediately that the
exact sequences of \cite[Th. 3]{quillen2}, used by Quillen to prove that the $K$-groups of $A$
are finitely generated when $A$ is a ring of $S$-integers in a global field, assemble to define
an exact couple, hence a spectral sequence converging to the homology of $KA$. This gives
potentially more power to Quillen's method, which as such yields no information on the ranks of
these $K$-groups.

A natural way to imagine such a spectral sequence is to consider the maps
\[BQ_{n-1}\sP(A)\to BQ_n\sP(A)\]
used by Quillen as \emph{homotopy cofibrations} rather than homotopy fibrations. The
resulting rank spectral sequence coincides with the above-mentioned spectral sequence: this is not immediately obvious, but follows
from a  simple argument of Vogel, see Remark \ref{r1}.

The aim of this note is to construct the rank spectral sequence in a way as functorial as possible. The two operative ingredients are Thomason's theorem on the nerve of a Grothendieck construction (Theorem \ref{t1}) and the notion of \emph{cellular functor} (Definition \ref{cell}), which is well-adapted to the present context thanks to Theorem \ref{p1}. After the first version of this paper was written, Fei Sun observed that Theorem \ref{p1} may also be used to recover part of Quillen's proof of the Solomon-Tits theorem in \cite{quillen2}. We complete Sun's remark in \S \ref{SolTits}, by covering the missing part: unlike in Quillen's argument, Theorem A is not used there. In this light, one might think of Theorem \ref{p1} as ``dual'' to Theorem A, potentially playing a similar r\^ole for homotopy cofibrations as Theorem A plays for homotopy fibrations.

The main theorem is Theorem \ref{t4}: we apply the theory in the slightly more general case of torsion-free sheaves over an integral scheme, which might be useful elsewhere.

The next step is to compute the $d^1$ differentials of this rank spectral sequence. This has been done by Fei Sun in his thesis, using the \emph{universal modular symbols} of Ash-Rudolph \cite{AR}: see \cite{sam} and  \cite{fei}.

\subsection*{Acknowledgements} I thank Pierre Vogel for explaining me the argument of Remark \ref{r1} in 2008, and Georges Maltsiniotis for pointing out Thomason's paper \cite{thomason}. I also thank Fei Sun for several discussions around these questions and Jo\" el Riou for helpful comments, especially for pointing out a wrong axiomatisation of Proposition \ref{pcontr}  in a previous version of this manuscript.

I don't think Theorem A is used explicitly anywhere. Yet I feel its spirit is prevalent in this text, hence the title.

\subsection*{Notation} We denote by $\Set$, $\Ab$, $\sSet$, $\Cat$ the category of (small) sets, abelian groups,
simplicial sets, categories. For $n\ge 0$,
$[n]$ denotes the totally ordered set $\{0,\dots, n\}$, considered as a small category. We
write $*$ for the category with one object and one morphism (sometimes for the set with $1$
element). Finally, $\bDelta$ denotes the category of simplices (objects: finite nonempty
ordinals, morphisms: non-decreasing maps).

We shall use Mac Lane's comma notation \cite[p. 47]{mcl}: if we have a diagram of functors
\[\sC\by{F} \sD\yb{F'} \sC'  \]
the comma category $F\downarrow F'$ has for objects the diagram $F(c)\to F'(c')$, and for
morphisms the obvious commutative diagrams. We use the following abbreviations: if $c\in \sC$,
yielding the functor $F_c:*\to \sC$, we write $F_c\downarrow F' =c\downarrow F'$; similarly on
the right.

\section{Nerves with coefficients}

Subsections \ref{s1.1}--\ref{s1.3} can essentially be found in Goerss-Jardine \cite[Ch. IV]{gj}.

\subsection{Set-valued coefficients}\label{s1.1}
\begin{para}\label{1} Let $\sD\in \Cat$ be a small category. The \emph{nerve} of $\sD$ is the
simplicial set $N(\sD)$ with
\[N_n(\sD)= Ob \Cat([n],\sD)= \coprod_{d_0\to \dots \to d_n} *, \; d_i\in \sD\]
cf. \cite[II, 4.1]{GZ}.
\end{para}

\begin{para}\label{2} Let $\sD\in \Cat$ and let $F:\sD\to \Set$ be a covariant functor. The
\emph{nerve of $\sD$ with coefficients in $F$} is the simplicial set $N(\sD,F)$ with
\[N_n(\sD,F) = \coprod_{d_0\to \dots \to d_n} F(d_0), \; d_i\in \sD\]
cf. \cite[App. II, 3.2]{GZ}. For $F$ the constant functor with value $*$, we recover the nerve
of $\sD$.
\end{para}

\begin{para}\label{3} Let $(\sD,F)$ be as in \ref{2}. We have the associated category
\[[\sD,F]=\{(d,x)\mid d\in \sD, x\in F(d)\}\]
where a morphism $(d,x)\to (d',x')$ is a morphism $f\in \sD(d,d')$ such that $F(f)(x)=x'$.
This category has two other equivalent descriptions:
\begin{enumerate}
\item \cite[II, 1.1]{GZ} Let $y=\sD^{op}\to \Cat(\sD,\Set)$ be the ``coYoneda" embedding: then
$[\sD,F]\simeq y\downarrow F$, where $F$ is considered as an object of $\Cat(\sD,\Set)$.
\item $[\sD,F]\simeq *\downarrow F$, where $*\in \Cat(\sD,\Set)$ is the constant functor with
value $*$ and $F$ is considered as a functor.
\end{enumerate}
\end{para}

The following lemma is obvious:

\begin{lemma}\label{l1} There is a canonical isomorphism: $N(\sD,F)\simeq N([\sD,F])$.
\end{lemma}

\subsection{Abelian group-valued coefficients}\label{s1.2}

\begin{para} Suppose that $F$ takes its values in the category $\Ab$ of abelian groups. Then
$N(\sD,F)$ is a simplicial abelian group; it has \emph{homology groups} \cite[App. II, 3.2,
p. 153]{GZ}
\[H_i(\sD,F)=H_i([N(\sD,F)]) =\pi_i(N(\sD,F);0)\]
where $[X]$ is the chain complex associated to a simplicial abelian group $X$ by taking for
differentials the alternating sums of the faces.
\end{para}
\enlargethispage*{20pt}

\begin{para}\label{ez} Let us recall the Eilenberg-Zilber--Cartier theorem as expounded in
\cite[p. 7]{illusie} (see also \cite[2.9 and 2.16]{dp}). To a bisimplicial abelian group $X$,
one may associate a double complex $[X]$ as in the simplicial case. To a bisimplicial object
$X$ one may associate the diagonal simplicial object $\delta X$:
\[(\delta X)_n=X_{n,n}\]
and to a double complex $C$ one may associate the total complex $\Tot C$:
\[(\Tot C)_n = \bigoplus_{p+q=n} C_{p,q}.\]
Then, given the diagram of functors
\[\begin{CD}
(\bDelta\times \bDelta)^\op \Ab @>[\ ]>> C_{++}(\Ab)\\
@V{\delta}VV @V{\Tot}VV\\
\bDelta^\op \Ab @>[\ ]>> C_+(\Ab)
\end{CD}\]
there exist two natural transformations
\begin{align*}
\text{shuffle}_X:& \Tot  [X]\to [\delta X]\\
\text{Alexander-Whitney}_X:& [\delta X]\to \Tot  [X]
\end{align*}
which are quasi-inverse homotopy equivalences. 
\end{para}

\subsection{Simplicial set-valued coefficients}\label{s1.3}
\begin{para} If $X$ is a simplicial set, we may associate to it the free simplicial abelian
group $\Z X$ generated by $X$, with $(\Z X)_n=\Z X_n$. Similarly with a bisimplicial set. The
\emph{homology} of $X$ is the homotopy of $\Z X$, or equivalently the homology of $[\Z X]$.
Similarly with coefficients in an abelian group $A$, using $\Z X\otimes A$.

We shall usually write $[\Z X]=:C_*(X)$; if $X =N(\sD)$ for a category $\sD$, we abbreviate $C_*(X)$ into $C_*(\sD)$.
\end{para}

\begin{para}\label{sset} Let $\sD\in \Cat$ and let $\F:\sD\to \sSet$ be a functor. We may
generalise the definition of \ref{2} to get a \emph{bisimplicial set} $N(\sD,\F)$:
\[N_{p,q}(\sD,\F)=\coprod_{d_0\to \dots \to d_p} \F_q(d_0), \; d_i\in \sD.\]
We may then take the diagonal $\delta N(\sD,\F)$, which is a simplicial set. We define
\begin{align*}
\pi_i(\sD,\F;(d_0,x_0))&= \pi_i(\delta N(\sD,\F);(d_0,x_0))\\
C_*(\sD,\F;A) & = C_*(\delta N(\sD,\F))\otimes A\\
H_i(\sD,\F;A) &= H_i(C_*(\sD,\F;A))
\end{align*}
for $d_0\in \sD_0$ and $x_0\in F_0(d_0)$ a chosen base point, and for $A$ an abelian group of
coefficients.
\end{para}

\begin{lemma}\label{l6} a) Let $X=(X_{p,q}), Y=(Y_{p,q})$ be two bisimplicial sets and $\phi:X\to Y$ a bisimplicial map. Suppose that, for each $p\ge 0$, $\phi_{p,*}:X_{p,*}\to Y_{p,*}$ is a weak equivalence. Then $\delta \phi:\delta X\to \delta Y$ is a weak equivalence.\\
b) Let $\F,\G:\sD\to \sSet$ be two functors, and let $\phi:\F\to \G$ be a morphism of functors. Suppose that, for each $d\in \sD$, $\phi(d):\F(d)\to \G(d)$ is a weak equivalence. Then $\delta N(\sD,\phi):\delta N(\sD,\F)\to \delta N(\sD,\G)$ is a weak equivalence.
\end{lemma}

\begin{proof} a) is well-known (for example, it is the special case of the Bous\-field-Friedlander theorem of \cite[Th. 4.9 p. 229]{gj} where $Y=W=*$), and b) follows from a).
\end{proof}

\begin{para} For an abelian group $A$ and for $q\ge 0$, we may consider the additive functor
\begin{align*}
H_q(\F,A):\sD&\to \Ab\\
d&\mapsto H_q(\F(d),A).
\end{align*}
This gives a meaning to:
\end{para}

\begin{lemma}\label{l4} There is a spectral sequence
\[E^2_{p,q}=H_p(\sD,H_q(\F,A))\Rightarrow H_{p+q}(\sD,\F;A)\]
\end{lemma}

\begin{proof} We shall use \ref{ez}: it implies that $[\delta \Z N(\sD,\F)]$ is homotopy
equivalent to $\Tot [\Z N(\sD,\F)]$. Therefore
\begin{multline*}
H_*(\sD,\F;A) := H_*([\delta \Z N(\sD,\F)]\otimes A)\\
\simeq H_*(\Tot [\Z N(\sD,\F)]\otimes A)=
H_*(\Tot\bigoplus_{p,q} \bigoplus_{d_0\to\dots\to d_p} [\Z F_q(d_0)]\otimes A).
\end{multline*}

Consider the first spectral sequence associated to this double
complex in Cartan-Eilenberg \cite[Ch. XV, \S 6]{ce}: the formula (1) of loc. cit, p. 331 shows
that it is the desired spectral sequence.
\end{proof}

\subsection{Category-valued coefficients} (See also \cite[IV.3]{weibel}.)

\begin{para} Let $\sD\in \Cat$ and let $\F:\sD\to \Cat$ be a functor. Composing $\F$ with the
nerve functor, we get a functor $N\circ \F:\sD\to \sSet$, hence a bisimplicial set as in 
\ref{sset}
\[N(\sD,\F)= N(\sD,N(\F)).\]
\end{para}

\begin{para}\label{cat} We now extend the construction in \ref{3}. This yields the category
$\sD\int \F$ (Grothendieck construction, \cite[Exp. VI, \S\S 8,9]{sga1}):
\begin{itemize}
\item Objects are pairs $(d,x)$, $d\in Ob(\sD), x\in Ob (\F(d))$.
\item For two objects $(d,x), (d',x')$, a morphism $(d,x)\to (d',x')$ is a morphism
$f:d\to d'$ and a morphism $g:\F(f)(x)\to x'$.
\item For three objects $(d,x),(d',x'), (d'',x'')$ and two morphisms $(f,g):(d,x)\to
(d',x')$, $(f',g'):(d',x')\to (d'',x'')$, $(f',g')\circ (f,g) := (f'\circ f, g'\circ
\F(f')(g))$.
\end{itemize}

The Grothendieck construction is covariant in $\F$. In particular, there is a canonical functor
$\sD\int\F\to \sD$ induced by the morphism $\F\to *$, where $*$ is the constant functor with
value the point category. We call this functor the \emph{augmentation}.

Lemma \ref{l1} then generalises as
\end{para}

\begin{thm}[Thomason \protect{\cite[Th. 1-2]{thomason}}]\label{t1} There is a canonical weak
equivalence 
\[\delta N(\sD,\F)\to N(\sD\int \F)\]
sending a cell $(d_0\by{f_1}\dots\by{f_n} d_n,x_0\by{g_1}\dots\by{g_n} x_n)$ ($x_i\in F(d_0)$)
to the cell
\[\left((d_0,y_0)\by{h_1}\dots\by{h_n} (d_n, y_n)\right)\quad (y_i\in
F(d_i))\] 
with $y_i:= F(f_i\dots f_1)x_i $ and $h_i = (f_i,F(f_i\dots f_1)g_i)$.\qed
\end{thm}

\begin{para}\label{adj}
Let $T:\sC\to \sD$ be a functor between two small categories. To $T$ we
associate the functor $\F_T:\sD\to \Cat$ sending $d$ to $T\downarrow d$. The category
$\sD\int\F_T$ may be identified with the comma category $T\downarrow Id_\sD$. Therefore
there are $3$ functors:
\[p_1:\sD\int \F_T\to \sC,p_2:\sD\int \F_T\to \sD, s:\sC\to \sD\int \F_T\]
with
\[p_1([T(c)\by{\phi} d])=c,p_2([T(c)\by{\phi} d])=d, s(c)=[T(c)=T(c)].\]
\end{para}
We have:

\begin{lemma}\label{l3} $s$ is left adjoint to $p_1$. Hence $s$ and $p_1$ induce quasi-inverse
homotopy equivalences
$N(\sC)\simeq N(\sD\int \F_T)$.\qed
\end{lemma}

From Theorem \ref{t1} and Lemma \ref{l4}, we deduce

\begin{cor}[cf. \protect{\cite[App. II]{GZ}}]\label{c1} There is a spectral sequence
\[E^2_{p,q} = H_p(\sD,H_q(\F_T,A))\Rightarrow H_{p+q}(\sC,A)\]
for any abelian group $A$.\qed
\end{cor}

\section{A long homology exact sequence}

\subsection{Spectral sequences and exact couples}

\begin{para}\label{2.1.1} Let $\sT$ be a triangulated category with countable direct sums, and
let $C_0\by{i_1} \dots \by{i_n} C_n\by{i_{n+1}}\dots$ be a sequence of objects of $\sT$. Let
$C$ be a homotopy colimit (mapping telescope) of the $C_n$ \cite{bok-nee}. Let $H:\sT\to \sA$ be a
(co)homological functor to some abelian category $\sA$: we assume that
$H$ commutes with countable direct sums.  (Alternately, we could refuse infinite direct sums
and assume that $i_n$ is an isomorphism for $n$ large enough.) To get an associated spectral
sequence, the simplest is the technique of exact couples \cite[pp. 152--153]{geomcell}: for each $n$, choose a cone
$C_{n/n-1}$ of $f_n$, so that the exact triangles
\[C_{p-1}\by{i_p} C_p\by{j_p} C_{p/p-1} \by{k_p} C_{p-1}[1]\]
yield long homology exact sequences
\[\dots H_n(C_{p-1})\by{i_{p,n}} H_n(C_p)\by{j_{p,n}} H_n(C_{p/p-1})\by{k_{p,n}} H_{n-1}(C_{p-1})\dots\]
where $H_n(X):=H(X[-n])$. The exact couple defined by 
\[D_{p,q} = H_{p+q}(C_p),\quad E_{p,q} = H_{p+q}(C_{p/p-1})\]
and the relevant $i,j,k$ define a spectral sequence abutting to $H_{p+q}(C)$. The $E^1$-terms of this spectral sequence are simply $E^1_{p,q}=E_{p,q}$, and $d^1_{p,q} = j_{p-1,p+q-1}k_{p,p+q}$ is induced by $j_{p-1}[1]k_p$.

Here is a more concrete description of this differential:
\end{para}

\begin{lemma}\label{l2.1} Let $C_{p/p-2}$ be a cone for $i_pi_{p-1}$, so that we may obtain a commutative diagram
\[\begin{CD}
C_{p-2} @= C_{p-2}  \\
@Vi_{p-1}VV @V{i_pi_{p-1}}VV\\
C_{p-1} @>i_p>> C_p @>j_p>> C_{p/p-1}@>k_p>> C_{p-1}[1] \\
@Vj_{p-1}VV @VVV @V{=}VV@Vj_{p-1}[1]VV\\
C_{p-1/p-2} @>\bar i_p>> C_{p/p-2} @>\bar j_p>> C_{p/p-1} @>\bar k_p>> C_{p-1/p-2}[1] 
\end{CD}\]
of exact triangles (from the suitable axiom of triangulated categories). Then $d^1_{p,q}$ is the boundary map $\bar k_{p,n}$ in the long exact sequence
\begin{multline*}
\dots H_{p+q}(C_{p-1/p-2})\by{\bar i_{p,n}} H_{p+q}(C_{p/p-2})\\
\by{\bar j_{p,n}} H_{p+q}(C_{p/p-1})\by{\bar k_{p,n}} H_{p+q-1}(C_{p-1/p-2})\dots
\end{multline*}
\end{lemma}

\begin{proof} The diagram shows that $\bar k_p = j_{p-1}[1]\circ k_p$.
\end{proof}

\begin{para}\label{2.1.3} In the usual case of a filtered complex $C_p=F_p C$, we may of course choose $C_{p/p-1}=F_pC/F_{p-1} C$ and $C_{p/p-2}=F_p C/F_{p-2} C$.
\end{para}

\begin{para}\label{2.1.4} Let $\sQ_0\by{T_1} \sQ_1\by{T_2} \dots\to \sQ_n\to\dots$ be a sequence of
categories and functors. Let $\sQ= \colim \sQ_n$. (Since there are no natural transformations involved this is a na\"\i ve colimit, defined objectwise and morphismwise.) Considering the corresponding sequence of chain complexes of nerves
\[C_*(\sQ_0)\by{i_1} C_*(\sQ_1)\dots\]
we get from the yoga of \ref{2.1.1} a spectral sequence abutting to $H_*(\sQ)$ (possibly with coefficients).

If the functors $T_n$ are faithful and injective on objects, the maps $i_n$ are injective and we are in the simpler situation of \ref{2.1.3}.
\end{para}

\subsection{Unreduced and reduced homology}

\begin{para}\label{red1} Let $(X,x)$ be a pointed simplicial set. The \emph{reduced homology} of $X$ with coefficients in an abelian group $A$ is
\begin{multline*}
\tilde H_i(X,A) = H_i(X,x;A) := \Coker(H_i(x,A)\to H_i(X,A))\\ (=H_i(X,A) \text{ if } i\ne 0).
\end{multline*}

This definition apparently depends on the choice of $x$: if we don't want to choose a basepoint, we may alternately define
\[\tilde H_i(X,A)=\Ker(H_i(X,A)\to H_i(*,A)).\]

Any choice of $x\in X_0$ will split the map $X\to *$, realising $\tilde H_i(X,A)$ as the above-described direct summand of $H_i(X,A)$.

A homotopy cofibre sequence $X\to Y\to Z$ is equivalent to a homotopy cocartesian square
\[\begin{CD}
X@>>> Y\\
@VVV @VVV\\
*@>>> Z
\end{CD}\]
hence the corresponding long exact homology sequence may be written (via Mayer-Vietoris!) as
\[\dots H_i(X,A)\to H_i(Y,A)\to \tilde H_i(Z,A)\to H_{i-1}(X,A)\to\dots\]
\end{para}

\begin{para}\label{red2} Suppose $\F:\sD\to \sSet$ is a functor; suppose that $\F(d)\neq
\emptyset$ for any $d\in \sD$. We then define
\[C_*(\sD,\tilde \F;A)= \Ker\left(C_*(\sD,\F)\to C_*(\sD)\right)\otimes A\]
where the last map is induced by the natural transformation $\F\to *$, and
\[H_i(\sD,\tilde \F;A)= H_i(C_*(\sD,\tilde \F;A)).\]

(Here the map $\F\to *$ is not necessarily split, so we have to be more careful. We think of
$\tilde \F$ as a desuspension of the homotopy cofibre of $\F\to *$.)

More generally, if $\F\to \G$ is a morphism of functors, we define
\[C_*(\sD,\F\to \G;A) = cone\left(C_*(\sD,\F)\to C_*(\sD,\G)\right)[-1].\]

If $\G=*$, then $C_*(\sD,\F\to *;A)$ is homotopy equivalent to $C_*(\sD,\tilde \F;A)$ under the
nonemptiness assumption on $\F$.
\end{para}

\begin{para} Suppose $\F:\sD\to \Cat$ is a functor. We have the projection $\F\to *$, where
$*$ is the constant functor with values the category with $1$ onject and $1$ morphism. As in
\ref{red2} we define
\begin{align*}
C_*(\sD,\tilde \F;A)&= \Ker\left(C_*(\sD,\F)\to C_*(\sC)\right)\otimes A\\
H_i(\sD,\tilde \F;A)&= H_i(C_*(\sD,\tilde \F;A)).
\end{align*}
provided $\F(d)\neq \emptyset$ for any $d\in \sD$. In general, we define $H_*(\sD,\tilde\F;A)$
as the homology of $C_*(\sD,\F\to *;A)$ as before.
\end{para}

\subsection{Cellular functors}

\begin{lemma}\label{l10} Let $\F_T$ be as in \ref{adj}. Suppose $T$ fully faithful. Then, for any $c\in \sC$, $\F_T(T(c))$ has a final object.
\end{lemma}

\begin{proof} Such a final object is given by $[T(c)=T(c)]$. 
\end{proof}

\begin{defn}\label{cell} Let $T:\sC\to \sD$ be a functor. We say that $T$ is \emph{cellular} if
\begin{itemize}
\item $T$ is fully faithful.
\item For any $d\in \sD-\sC$ and any $c\in \sC$, $\sD(d,c)=\emptyset$.
\end{itemize}
\end{defn}

If $T$ is cellular, then it defines a ``stratification" $\sD = \sC\coprod (\sD-\sC)$ in the
following sense:

\begin{lemma}\label{l2.2} Let $T:\sC\to \sD$ be a cellular functor; let $F_1,F_2:\sD\to \sE$ be two functors and let $\phi:F_1\Rightarrow F_2$ be a natural transformation. Then there is a unique factorisation of $\phi$ as a composition
\[F_1\overset{\phi_1}{\Rightarrow}F_\phi\overset{\phi_2}{\Rightarrow}F_2\]
such that $F_\phi(d)=F_2(d)$, $\phi_2(d)=1_{F_2(d)}$ for $d\in \sD-T(\sC)$, and  $F_\phi T(c)=F_1T(c)$, $\phi_1(T(c))=1_{F_1T(c)}$ for $c\in \sC$.
\end{lemma}

\begin{proof} 
For simplicity, we drop the notation $T$ in this proof. Let us define a functor structure on $F_\phi$ as follows: if $f:d\to d$ is a morphism in $\sD$, then three cases may occur:
\begin{itemize}
\item $d,d'\in \sD-\sC$. We define $F_\phi(f)$ as $F_2(f)$.
\item $d\in \sC$, $d'\in \sD-\sC$. We define $F_\phi(f)$ as $\phi_d F_1(f)=F_2(f)\phi_c$.
\item $d,d'\in \sC$. We define $F_\phi(f)$ as $F_1(f)$. 
\end{itemize}

Similarly, we define $\phi_1(d)$ as $\phi(d)$ for $d\in \sD-\sC$ and $\phi_2(c)=\phi(c)$ for $c\in \sC$.

These definitions are the only possible ones if the lemma is to be correct. Checking that $F_\phi$ respects composition of morphisms and that $\phi_1,\phi_2$ do define natural transformations is a matter of case-by-case bookkeeping.
\end{proof}

\begin{prop}\label{p01} Let $T:\sC\to \sD$ and $\phi:F_1\Rightarrow F_2$ be as in Lemma \ref{l2.2}, with $\sE=\sSet$, and let $\sD-\sC$ be the full subcategory of $\sD$ given by the objects not
in $\sC$. Assume that $\phi_{T(c)}$ is a weak equivalence for any $c\in \sC$. Then the commutative diagram of bisimplicial sets
\[\begin{CD}
N(\sD-\sC,F_1) @>>> N(\sD,F_1)\\
@V{\phi}VV @V{\phi}VV\\
N(\sD-\sC,F_2) @>>> N(\sD,F_2)
\end{CD}\]
is homotopy cocartesian,  \emph{i.e.} becomes so after applying the diagonal $\delta$. 
\end{prop}

\begin{proof} Applying Lemma \ref{l2.2}, we can enlarge the above diagram as follows:
\begin{equation}\label{eq2.1}
\begin{CD}
N(\sD-\sC,F_1) @>>> N(\sD,F_1)\\
@V{\phi_1}VV @V{\phi_1}VV\\
N(\sD-\sC,F_\phi) @>>> N(\sD,F_\phi)\\
@V{\phi_2}VV @V{\phi_2}VV\\
N(\sD-\sC,F_2) @>>> N(\sD,F_2)
\end{CD}
\end{equation}

It suffices to show that the top and bottom squares in \eqref{eq2.1} are both homotopy
cocartesian.

By hypothesis $\phi_2(d)$ is a weak equivalence for $d\in \sC$, and it is trivially a weak
equivalence for $d\in \sD - \sC$. Therefore, by Lemma \ref{l6} b), the two vertical maps in the bottom square of \eqref{eq2.1} become weak equivalences after applying $\delta$; \emph{a fortiori}, this bottom square is homotopy cocartesian.

On the top right of \eqref{eq2.1}, a typical term is
\[\coprod_{d_0\to \dots \to d_p} F_1(d_0).\]

We split this coproduct in two parts: one is where all the $d_i$ are in $\sD - \sC$ (call it $A_1$)
and the other is the rest (call it $B_1$). Similarly on the middle right of \eqref{eq2.1} (call them $A_\phi$ and
$B_\phi$).

Now the cellularity assumption implies that all
cells in $B_1$ and $B_\phi$ begin with $d_0\in \sC$. But for such a $d_0$, $F_1(d_0)\to F_\phi(d_0)$ is
a bijection. Thus $B_1\to B_\phi$ is bijective.  On the other hand, by definition of the cofibre,
$A_1$ and $A_\phi$ become a point in the cofibres. Thus the induced map on cofibres is bijective, and the top square of \eqref{eq2.1} is homotopy cocartesian as well.
\end{proof}

\begin{defn}\label{d2} A (naturally) commutative square of categories and functors
\emph{is homotopy cocartesian} if it is after applying the nerve functor.
\end{defn}

\begin{thm}\label{p1} Let $\sD-\sC$ be the full subcategory of $\sD$ given by the objects not
in $\sC$. If
$T$ is cellular, the naturally commutative diagram of categories
\[\begin{CD}
(\sD-\sC)\int\F_T @>{p}>> \sC\\
@V{\epsilon}VV @V{T}VV\\
\sD-\sC @>\iota >> \sD
\end{CD}\]
is homotopy cocartesian, where $\epsilon$ is the augmentation (see
\S\ref{cat}), $p$ is induced by the first projection $p_1$ of Lemma \ref{l3} and $\iota$
is the inclusion.
\end{thm}

\begin{proof} Note first the natural transformation
\begin{align*}
u:T\circ p &\Rightarrow \iota \circ \epsilon\\ 
u_{[T(c)\by{f} d]} &= f
\end{align*}
which explains ``naturally commutative" (here in the weak sense). By Theorem \ref{t1} and Lemma
\ref{l3}, it suffices to prove that the (commutative) diagram of bisimplicial sets
\[\begin{CD}
 N(\sD-\sC,\F_T) @>>> N(\sD,\F_T)\\
@VVV @VVV\\
N(\sD-\sC) @>>> N(\sD)
\end{CD}\]
is homotopy cocartesian, \emph{i.e.} becomes so after applying the diagonal $\delta$.  In view of Lemma \ref{l10}, this is a special case of Proposition \ref{p01}.
\end{proof}

\begin{cor}\label{t2} Let $T:\sC\to \sD$ be a cellular functor. Then the mapping cone of
\[C_*(T):C_*(\sC)\to C_*(\sD)\]
is quasi-isomorphic to $C_*(\sD-\sC,\tilde \F_T)[1]$ (cf. \ref{red2}). In particular, we
have a long exact sequence
\begin{multline*}
\dots\to H_i(\sD-\sC,\tilde \F_T;A)\to H_i(\sC,A)\\
\to H_i(\sD,A)\to  H_{i-1}(\sD-\sC,\tilde \F_T;A)\to \dots
\end{multline*}
for any abelian group $A$.
\end{cor}

\begin{proof} This follows from Theorems \ref{p1} and \ref{t1}.
\end{proof}

\subsection{Cellular filtrations}

\begin{thm}\label{t3} Let $\sQ_1\to \sQ_2\to \dots\to \sQ_n\to\dots\to \sQ$ be a sequence of
categories. We assume:
\begin{itemize}
\item The functors $T_n:\sQ_{n-1}\to \sQ_n$ are cellular (\ref{cell}).
\item $\sQ = \colim \sQ_n$.
\end{itemize}
Write $\F_n$ for $\F_{T_n}$. Then, for any abelian group $A$, there is a spectral sequence
of homological type
\[E^1_{p,q}=H_{p+q-1}(\sQ_p-\sQ_{p-1},\tilde \F_p;A) \Rightarrow H_{p+q}(\sQ,A).\]
\end{thm}

\begin{proof} This is the spectral sequence of \ref{2.1.4}, taking Corollary \ref{t2} into account. 
\end{proof}

\section{Cellular functors and the Solomon-Tits theorem}\label{SolTits}

\subsection{An abstract version of the Solomon-Tits theorem (for $GL_n$)} 
This section was catalysed by Fei Sun's insight that one can use Theorem \ref{p1} to understand part of Quillen's proof of the Solomon-Tits theorem in \cite{quillen2}. We start with an abstraction of his argument.

\begin{prop}[Sun, essentially]\label{pfei} Let $\V$ be an poset (considered as a category), $\H$ a subset of maximal elements in $\V$ and $\Y=\V-\H$. For any $W\in \V$, write $\V_W=\{X\in \V\mid X< W\}$. Then the naturally commutative diagram
\[\begin{CD}
\displaystyle\coprod_{H\in \H} \V_H @>j>> \Y\\
@V{k}VV @V{i}VV \\
\displaystyle\coprod_{H\in \H} * @>l>> \V
\end{CD}\]
is homotopy cocartesian. Here $i$ is the inclusion, $j$ is componentwise the inclusion, $k$ is the tautological projection and $l$ is the inclusion of the discrete set $\coprod\limits_{H\in \H} *=\H$ into $\V$; a natural transformation $ij\Rightarrow lk$ is defined by the inequality $W\le H$ for $W\in \V_H$.
\end{prop}

\begin{proof} The hypotheses imply that $i$ is a cellular functor, so it suffices to compute $\H\int \F_i=\coprod\limits_{H\in \H} i\downarrow H$. Clearly, $i\downarrow H=\V_H$.
\end{proof}

Sun's second insight was that one can replace the simplicial complex $Y$ in the proof of Quillen's Claim in \cite[\S 2]{quillen2} by the poset of its vertices. We make use of this observation now.

\begin{prop} \label{pcontr} Keep the notation of Proposition \ref{pfei}. Assume that there exists $L\in \V$ such that
\begin{thlist}
\item $\H= \{W\in \V\mid W \text{ maximal and } L\nleq W\}$.
\item For any $W\in \Y$, the supremum $L\vee W$ exists in $\V$.
\end{thlist}
Then $\Y$ is contractible.
\end{prop}

\begin{proof} It is directily inspired by Quillen's proof of his claim (\emph{loc. cit.}), but avoids the use of Theorem A. 

For $W\in \Y$, $L\vee W$ cannot be in $\H$ by (i), hence $\phi(X)=L\vee X$  defines an endofunctor $\phi:\Y\to \Y$, and the inequality $X\le L\vee X$ defines a natural transformation $Id_\Y\Rightarrow \phi$. Now  
\[\phi(\Y)=\{X\in \Y\mid L\le X\}\]
has the minimal element $L$, hence is contractible. Thus $\phi$ factors through a contractible poset; since $\phi$ is homotopic to $Id_\Y$, this proves the claim.
\end{proof}

\begin{cor}\label{c3} Under the assumptions of Proposition \ref{pcontr}, there is a zigzag of homotopy equivalences
\[\bigvee_{H\in \H} \Sigma N(\V_H)\iso N(\V)/N(\Y)\osi N(\V).\]
\end{cor}

\begin{proof} For any poset $X$, write $\bar X$ for the poset $X\cup \{+\}$, where $+$ is a new element larger than all elements in $X$: thus $\bar X$ is contractible. Consider the strictly commutative diagram of functors
\[\begin{CD}
\displaystyle\coprod_{H\in \H} \V_H @>j>> \Y\\
@V{i'}VV @V{i}VV \\
\displaystyle\coprod_{H\in \H} \bar \V_H @>\bar\jmath>> \V\\
@A{\epsilon}AA || \\
\displaystyle\coprod_{H\in \H} * @>l>> \V.
\end{CD}\]

Here $i'$ is componentwise the natural inclusion, $\bar\jmath$ sends $\V_H$ to itself and $+_H$ to $H$ while $\epsilon$ sends $*_H$ to $+_H$. The latter functor has a retraction $\bar k$ which extends the functor $k$ of Proposition \ref{pfei}.  From this and the latter proposition, we deduce that the top square is homotopy cocartesian. We deduce a homotopy equivalence
\[\bigvee_{H\in \H} N(\bar \V_H)/N(\V_H)\iso N(\V)/N(\Y)\]
induced by $\bar\jmath$. But $N(\bar \V_H)/N(\V_H)$ is canonically equivalent to $\Sigma N(\V_H)$; the conclusion then follows from Proposition \ref{pcontr}.
\end{proof}

\subsection{From abstract to concrete}\label{soltits} Given a poset $\V$, write $\Simpl(\V)$ for the abstract simplicial complex associated to $\V$: the simplices of $\Simpl(\V)$ are by definition the finite totally ordered subsets of $\V$.

Let $V$ be a finite-dimensional vector space over a [possibly skew] field $k$.  Its \emph{Tits building} $\T(V)$ is [canonically isomorphic to] the flag complex of $V$: elements of $\T(V)$ are flags of proper subspaces of $V$ \cite[\S 2]{quillen2}. It follows that $\T(V)=\Simpl(\V)$, where $\V$ is the poset of proper subspaces of $V$. Thus $\T(V)=\emptyset$ if $\dim V\le 1$; we now assume $\dim V\ge 2$.
Choosing a line $L\subset V$, we get from Corollary \ref{c3} Quillen's decomposition
\[\T(V)\sim \bigvee_{H\in\H} \Sigma \T(H)\]
where $\H$ is the set of hyperplanes in $V$ which do not contain $L$. It follows inductively that $\T(V)$ has the homotopy type of a bouquet of $(n-2)$-spheres.

\section{The rank spectral sequence}

\subsection{$K$-theory of schemes} The first example of application of Theorem \ref{t3} is to
Quillen's $Q$-construction $\sQ(X)$ on the exact category of locally free sheaves of finite
rank over a scheme $X$. Let $\sQ_n=\sQ_n(X)$ be the full subcategory of $\sQ(X)$ consisting of
locally free sheaves of rank $\le n$. Then the assumptions of Theorem \ref{t3} are satisfied
because, in $\sQ(X)$, there are no morphisms from a locally free sheaf of rank $n$ to a locally
free sheaf of rank $<n$. The resulting spectral sequence may be called the \emph{rank spectral
sequence} (for the homology of $\sQ(X)$). 

Note that $\sQ_n-\sQ_{n-1}$ is a groupoid, hence we get
\begin{equation}\label{eq8}
E^1_{p,q}=\bigoplus_{E_\alpha} H_{p+q-1}(\Aut(E_\alpha),\tilde \F_p)
\end{equation}
where $E_\alpha$ runs through the set of isomorphism classes of locally free sheaves of rank
$p$. We took coefficients $\Z$, for simplicity.

\subsection{$K'$-theory of integral schemes} Let $X$ be an integral scheme,
 with function field $K$. If $\eta=\Spec K$ is the generic point of $X$, we have the inclusion $j=\eta\to X$. If $E$ is a sheaf of $\sO_X$-modules, we write $E_K$ for $j^*E$.

\begin{defn} a) A coherent sheaf $E$ on $X$ is \emph{torsion-free} if the map $E\to j_*E_K$ is a monomorphism.\\
b) A subsheaf  $E'$ of a coherent sheaf $E$ is \emph{pure} if $E/E'$ is torsion-free.
\end{defn}

\begin{para} Inside the exact category of coherent sheaves, the full subcategory of torsion-free sheaves is closed under extensions and subobjects. A monomorphism $E'\to E$ between torsion-free sheaves is admissible (within the exact category of torsion-free sheaves) if and only if $E'$ is pure in $E$.
\end{para}

\begin{lemma}\label{l13} Let $\sQ^\coh(X)$ be Quillen's $Q$-construction on the category of coherent sheaves of $\sO_X$-Modules, and let $\sQ^\tf(X)$ be the full subcategory of torsion-free sheaves.
Then the inclusion $\sQ^\tf(X)\to \sQ^\coh(X)$ is a weak equivalence.
\end{lemma}

\begin{proof} The conditions of the resolution theorem \cite[Th. 3]{quillen} are verified since any locally free sheaf is torsion-free.
\end{proof}

\begin{prop}\label{p2} Let $E$ be a (coherent) torsion-free sheaf on $X$, with generic fibre
$E_K$. Then the map
\[F\mapsto F_K\]
defines a bijection from the set $Gr(E)$ of pure subsheaves of $E$ to the set $Gr(E_K)$ of
subvector spaces of $E_K$.
\end{prop}

\begin{proof} 
Let $V$ be a sub-vector space of $j^*E = E_K$. Define
\[E\cap V = E\times_{j_*j^*E} j_*V.\]

Then $E\cap V\in Gr(E)$, because the map 
$E/(E\cap V)\to  j_*j^*(E/(E\cap V))$ is a monomorphism (by definition of $E\cap V$). So we
have two maps:
\[j^*:Gr(E)\to Gr(E_K);\quad E\cap -:Gr(E_K)\to Gr(E).\]

We have
\[(E\cap V)_K = j^*(E\times_{j_*j^*E} j_*V)=j^*E\times_{j^*j_*j^*E}j^*j_*V=j^*E\times_{j^*E} V
= V\] so that $j^*\circ (E\cap -) = Id$. On the other hand, if $F$ is a pure subsheaf of $E$,
then
$F\subseteq E\cap j^*F$, and $j^*E=j^*(E\cap j^*F)$, hence (by exactness of $j^*$), $j^*(E\cap
j^*F/F)=0$. Thus $(E\cap j^*F)/F$ is a torsion subsheaf of the torsion-free sheaf $E/F$, hence
is $0$, and our two maps are inverse to each other.
\end{proof}

\begin{para}We write $\sQ_n^\tf(X)$ for the full subcategory of $\sQ^\tf(X)$ of torsion-free
sheaves $E$ such that $\dim_K E_K\le n$. We get another rank spectral sequence
\begin{equation}\label{eq8'}
E^1_{p,q}=\bigoplus_{E_\alpha} H_{p+q-1}(\Aut(E_\alpha),\tilde \F_p) \Rightarrow
H_{p+q}(\sQ^\coh(X))
\end{equation}
cf. Lemma \ref{l13}.
\end{para}

\begin{cor}\label{c2} Let $E\in \sQ_n^\tf(X)$, with generic fibre $E_K\in \sQ_n(K)$. Then the
functor
\[j^*:\sQ_{n-1}^\tf(X)\downarrow E\to \sQ_{n-1}(K)\downarrow E_K\]
is an equivalence of categories.
\end{cor}

\begin{proof} These categories are equivalent to the ordered sets of proper layers of torsion-free subsheaves of $E$ and $j^*E$ (compare \cite[top p. 102]{quillen}). Thus the result directly follows from Proposition
\ref{p2}.
\end{proof}

\begin{ex} Let $X$ be an integral Dedekind scheme (= noetherian, regular of Krull dimension
$\le 1$), with function field $K$. As is well-known, a coherent sheaf $F$ over a Dedekind
scheme is torsion-free if and only if it is locally free. Thus the above generalises the remark of \cite[pp. 191--192]{quillen2}.
\end{ex}

\subsection{The Tits building}

\begin{para}\label{3.2.3} In Corollary \ref{c2}, suppose $n\ge 2$. By \cite[Prop. p.
188]{quillen2}, the classifying space of the poset $\sQ_{n-1}(K)\downarrow E_K=J(E_K)$ is
$GL(E_K)$-weakly equivalent to the suspension of nerve of the Tits building of $E_K$, which in
turn is weakly equivalent to a wedge of $(n-2)$-spheres by the Solomon-Tits theorem (\cite[Th. 2
p. 180]{quillen2}, see \S \ref{soltits}). Hence $(\F_n)_{|E}$ is $\Aut(E)$-weakly equivalent to a wedge of
$(n-1)$-spheres.
\end{para}

\begin{para}\label{3.2.4} In Corollary \ref{c2}, suppose $n= 1$. Then $\sQ_{n-1}(K)\downarrow
E_K$ has two elements: $0\to E_K$ and $E_K\to 0$. Hence the conclusion of \ref{3.2.3} is still true.
\end{para}

\begin{thm}\label{t4} If $X$ is an integral scheme, then the $E^1$-terms of the rank spectral
sequence \eqref{eq8'} (with $\Z$-coefficients) are
\[E^1_{p,q} = \bigoplus_{E_\alpha} H_q(\Aut(E_\alpha), st(E_\alpha))\]
 where $E_\alpha$ runs through the isomorphism classes of torsion-free sheaves of rank $p$, and
$st(E_\alpha)= \tilde H_{p-1}((\F_p)_{|E_\alpha})$ is the [reduced] Steinberg module of $j^*E_\alpha$.
\end{thm}

\begin{proof} This follows from Corollary \ref{c2}, \ref{3.2.3}, \ref{3.2.4} and Lemma
\ref{l4}.
\end{proof}


\begin{rk}\label{r1} By an argument of Vogel, the exact sequences from Corollary \ref{t2} then
coincide with those of Quillen in \cite[Th. 3 p. 181]{quillen2}. In
general, consider a map $f:E\to B$ whose homotopy fibre $F$ has the homotopy type of a bouquet
of $n$-spheres, with $n>0$. So the Leray-Serre spectral sequence yields a long exact sequence
\begin{equation}\label{eq3.1}
\dots\to H_p(E)\to H_p(B)\to H_{p-n-1}(B,H_n(F))\to H_{p-1}(E)\to \dots
\end{equation}
 
If $C$ is the homotopy cofibre of $f$, we have another long exact sequence
\begin{equation}\label{eq3.2}
\dots\to H_p(E)\to H_p(B)\to \tilde H_p(C)\to H_{p-1}(E)\to \dots
\end{equation}
and we want to know that the two sequences coincide.

Here is Vogel's argument. We may assume that $f$ is a Serre fibtration. Let $E'$ be its mapping cylinder and $CF$ the cone over $F$, so
that we have a fibration of pairs
\[\begin{CD}
(CF,F) @>>> (E',E)\\
&& @VVV\\
&& (B,B).
\end{CD}\]

Since $H_q(CF,F)=\begin{cases} H_n(F)&\text{for $q=n+1$}\\ 0
&\text{else}\end{cases}$, the Leray-Serre spectral sequence for the pair
\[H_p(B,H_q(CF,F))\Rightarrow H_{p+q}(E',E)\]
yields isomorphisms 
\[\tilde H_p(C)\simeq H_p(E',E)\simeq H_{p-n-1}(B,H_{n+1}(CF,F))\simeq H_{p-n-1}(B,H_n(F))\]
which commute with the differentials of \eqref{eq3.1} and \eqref{eq3.2} by functoriality.
\end{rk}

\end{document}